\documentclass[a4paper]{amsart}
\usepackage{amsmath,amsthm,amssymb}
\usepackage[T1]{fontenc}
\usepackage{mathrsfs}
\usepackage{url}

\usepackage[pdftex]{graphicx}

\newtheorem{tw}{Theorem}
\newtheorem{lm}[tw]{Lemma}
\newtheorem{wn}[tw]{Corollary}
\newtheorem{pr}[tw]{Proposition}
\theoremstyle{definition}

\newtheorem{uw}{Remark}

\newcommand{\R}{\mathbb{R}}
\newcommand{\Z}{\mathbb{Z}}
\newcommand{\N}{\mathbb{N}}
\newcommand{\T}{\mathbb{T}}

\newcommand{\cS}{\mathcal{S}}
\newcommand{\C}{\mathbb{C}}
\newcommand{\Q}{\mathbb{Q}}
\newcommand{\cB}{\mathcal{B}}
\newcommand{\cC}{\mathcal{C}}
\newcommand{\cP}{\mathcal{P}}
\newcommand{\cT}{\mathcal{T}}

\newcommand{\cU}{\mathcal{U}}

\newcommand{\xbm}{(X,\mathcal{ B},\mu)}
\newcommand{\ycn}{(Y,\mathcal{ C},\nu)}

\newcommand{\vep}{\varepsilon}

\begin{document}
\date{\today}
\title{On the self-similarity problem for Gaussian-Kronecker flows}
\author{K. Fr\k{a}czek\and J. Ku\l aga\and M. Lema\'nczyk}
\address{Faculty of Mathematics and Computer Science, Nicolaus
Copernicus University, ul. Chopina 12/18, 87-100 Toru\'n, Poland}
\email{fraczek@mat.umk.pl, joanna.kulaga@gmail.com, mlem@mat.umk.pl}

\maketitle

\begin{abstract}
It is shown that a countable symmetric multiplicative subgroup
$G=-H\cup H$ with $H\subset\R_+^\ast$ is the group of
self-similarities of a Gaussian-Kronecker flow if and only if $H$
is  additively $\Q$-independent. In particular, a real number
$s\neq\pm1$ is a scale of self-similarity of a Gaussian-Kronecker
flow if and only if $s$ is transcendental. We also show that each
countable symmetric subgroup of $\R^\ast$  can be realized as the
group of self-similarities of a simple spectrum Gaussian flow
having the Foia\c{s}-Stratila property.
\end{abstract}

\section*{Introduction}

Assume that $\cT=(T_t)_{t\in\R}$ is a (measurable)
measure-preserving flow acting on a standard probability Borel
space $\xbm$. Given $s\in\R^\ast$, one says that it is a {\em
scale of self-similarity} of $\cT$ if $\cT$ is isomorphic to
$\cT_s:=(T_{st})_{t\in\R}$. Denote by $I(\cT)$ the set of all
scales of self-similarities of $\cT$. Then $\cT$ is called {\em
self-similar} if $I(\cT)\neq\{\pm1\}$. Classical examples of
self-similar flows are given by horocycle flows where $I(\cT)$
equals either $\R^\ast$ or $\R^\ast_+$ \cite{Ma}. A systematic
study of the problem of self-similarity has been done recently in
\cite{Da-Ry} and \cite{Fr-Le}. In particular,  $I(\cT)$  turns out
to be a multiplicative subgroup of $\R^\ast$ (\cite{Fr-Le}) which is
Borel (\cite{Da-Ry}), and one of the main problems in this domain is
to classify all Borel subgroups of $\R^\ast$ that may appear as
groups of self-similarities of ergodic flows; see also \cite{Da},
\cite{Ku}, \cite{Ry1}, \cite{Ry2} for a recent contribution to
other aspects of the problem of self-similarity of ergodic flows.
From this point of view the subclass of so called GAG flows
\cite{Le-Pa-Th}\footnote{In \cite{Le-Pa} as well as in
\cite{Le-Pa-Th} only Gaussian automorphisms are considered,
however all results can be rewritten for Gaussian flows.}  of the
class of Gaussian flows is especially attractive since
self-similarities appear there as natural invariants,
see~(\ref{ni}) below. By definition, GAG flows are those Gaussian
flows whose ergodic self-joinings remain Gaussian. All Gaussian flows
with simple spectrum are GAG flows \cite{Le-Pa-Th}. If
$\cT^\sigma=(T^\sigma_t)_{t\in\R}$ denotes the Gaussian flow
determined by a finite positive (continuous) measure $\sigma$ on
$\R_+$ and the flow is GAG then
\begin{equation}\label{ni}
\text{$I(\cT^\sigma)$ is equal to  the (multiplicative) group
$-I(\sigma)\cup I(\sigma)$},
\end{equation} where
$I(\sigma)=\{s\in\R_+^\ast:\: \sigma_s\equiv\sigma\}$ and
$\sigma_s=(R_s)_\ast(\sigma)$ denotes the image of $\sigma$ via
the map $R_s:x\mapsto sx$ \cite{Le-Pa-Th}. Recall that $-1$ is
always a scale of self-similarity for a Gaussian flow.

In this note we focus on the problem  of self-similarities in
some subclasses of simple spectrum Gaussian flows. We first
recall already  known results. Classically, if $\sigma$ is
concentrated on an additively $\Q$-independent Borel set
$A\subset\R_+$ then the Gaussian flow $\cT^\sigma$ has simple
spectrum, see \cite{Co-Fo-Si}. Moreover, the subgroup
$H:=I(\cT^\sigma)\cap\R^\ast_+$ is an additively $\Q$-independent
set. Indeed, suppose that $H$ is not an additively
$\Q$-independent set. That is, for some distinct
$h_1,\ldots,h_m\in H$ we have
\begin{equation}\label{km1}
\sum_{i=1}^m k_ih_i=0\quad \text{with}\quad
k_i\in\Z,\;i=1,\ldots,m\quad \text{and}\quad \sum_{i=1}^mk_i^2>0.
\end{equation}
Denote by $H_0\subset H$ the multiplicative subgroup generated by
$h_1,\ldots,h_m$. Since $H_0\subset I(\cT^\sigma)$, we have
$\sigma_{h}\equiv\sigma$ for $h\in H_0$, thus the Borel set
$B=\bigcap_{h\in H_0}hA$ has full $\sigma$-measure, is
$\Q$-independent, and is literally $H_0$-invariant. Take any
non-zero $x\in B$. Then the elements $h_ix\in B$, $i=1,\ldots,m$,
are distinct. Now,~(\ref{km1}) yields
\[
\sum_{i=1}^mk_i(h_ix)=x\sum_{i=1}^mk_ih_i=0,
\]
so $B$ is not independent, a contradiction.  On the other hand, in \cite{Fr-Le},
it is shown that whenever a countable group $H\subset \R^\ast_+$
satisfies:
\begin{equation}\label{mk2}\begin{array}{l}
\text{For each polynomial $P\in\Q[x_1,\ldots,x_m]$ if there is}\\
\text{a collection of distinct elements $h_1,\ldots, h_m$ in $H$
such that}\\ \text{$P(h_1,\ldots,h_m)=0$ then $P\equiv
0$,}\end{array}\end{equation} then there exists a probability
$\sigma$ concentrated on a Borel $\Q$-independent set such that
$I(\cT^\sigma)=-H\cup H$. It is not difficult to see that the
condition~(\ref{mk2}) is equivalent to saying that $H$ is an
additively $\Q$-independent set.

\begin{tw}[\cite{Fr-Le}]\label{km3}
Assume  that $G=-H\cup H$, where $H\subset \R^\ast_+$ is a
countable multiplicative subgroup. Then $G$ can be realized as
$I(\cT^\sigma)$ for a Gaussian flow whose spectral measure
$\sigma$ is concentrated on a Borel $\Q$-independent set if and
only if $H$ is an additively $\Q$-independent set.
\end{tw}

Note that for $H$ cyclic generated by $s\in\R_+$, the
$\Q$-independence condition   is equivalent to saying that $s$ is
transcendental. Hence, by Theorem~\ref{km3}, a real number $s$ can
be realized as a scale of self-similarity of a Gaussian flow
whose spectral measure is concentrated on a $\Q$-independent Borel
set if and only if $s$ is transcendental.

On the other hand, there are no restrictions on $H$ in the class
of all Gaussian flows having simple spectrum.

\begin{tw}[\cite{Da-Ry}]\label{sw}
For  each countable subgroup $H\subset\R^\ast_+$ there exists a
simple spectrum Gaussian flow $\cT^\sigma$ such that
$I(\cT^\sigma)=-H\cup H$.
\end{tw}

Note that, in particular, the above result of Danilenko and
Ryzhikov brings the positive answer to the open problem \cite{Le}
of existence of Gaussian flows $\cT^\sigma$ with simple spectrum
such that $\sigma$ is not concentrated on a $\Q$-independent set;
indeed, whenever $H$ is not an additively $\Q$-independent set, by
Theorem~\ref{km3}, the spectral measure $\sigma$ resulting from
Theorem~\ref{sw} cannot be concentrated on a Borel
$\Q$-independent set. See also \cite{Da} for constructions  of
Gaussian flows with zero entropy and having uncountable groups of
self-similarities.

Our aim is to continue investigations on realization of countable
subgroups as the groups of self-similarities in further restricted
classes of Gaussian flows whose spectral measures are classical
from the harmonic analysis point of view.  Recall some basic
notions. For every $s\in\R$ let  $\xi_s:\R\to {\mathbf S^1}$ be
given by $\xi_s(t)=\exp(2\pi i st)$. A finite positive Borel
measure $\sigma$ on $\R$ is called Kronecker if for each $f\in
L^2(\R,\sigma)$, $|f|=1$ $\sigma$-a.e., there exists a sequence
$(t_n)\subset\R$, $t_n\to \infty$, such that
\begin{equation}\label{kro1}
\xi_{t_n}\to f\quad \text{in}\quad L^2(\R,\sigma).
\end{equation}
Each measure $\sigma$ concentrated on a Kronecker set  \cite{Ko},
\cite{Li-Po} is a Kronecker measure. Indeed, Kronecker sets are
compact subsets of $\R$ on which each continuous function of
modulus one is a uniform limit of characters. Kronecker sets are
examples of $\Q$-independent sets \cite{Li-Po}. In general, as
shown in \cite{Le-Pa}, a Kronecker measure is concentrated on a
Borel set which is the union of an increasing sequence of
Kronecker sets, hence a Kronecker measure is concentrated on a
Borel $\Q$-independent set, and the restriction on $H$ in
Theorem~\ref{km3} applies.  This turns out to be the only
restriction as the main result of the note shows.

\begin{tw}\label{km10}
Assume that $G=-H\cup H$, where $H\subset \R^\ast_+$ is a
countable multiplicative subgroup. Then $G$ can be realized as
$I(\cT^\sigma)$~\footnote{In a sense, we can also control the
flows $\cT^{\sigma_s}$ for $s\notin-H\cup H$; we will prove their
disjointness from $\cT^\sigma$, see the proof of this theorem.}
for a Gaussian-Kronecker flow if and only if $H$ is an additively
$\Q$-independent set. In particular, $h\in\R_+$ can be a scale of
self-similarity for a Gaussian-Kronecker flow if and only if $h$
is transcendental.
\end{tw}

An extremal  case when two dynamical systems are non-isomorphic is
the disjointness in the Furstenberg sense \cite{Fu}, see also
\cite{Gl}, \cite{Ka-Th}, \cite{Le}, \cite{delaRu} for disjointness
results in ergodic theory. We would like also to emphasize that
the notion of disjointness turned out to be quite meaningful in
the problem of non-correlation with the M\"obius function of
sequences of dynamical origin \cite{Bo-Sa-Zi}: we need that an
automorphism $T$ has the property that $T^p$ and $T^q$ are
disjoint for any two different primes. In connection with
that we will prove the following.

\begin{tw}\label{mkj}
Assume that $\cT^\sigma=(T^\sigma_t)_{t\in\R}$ is a
Gaussian-Kronecker flow. If $s\in\Q\setminus\{\pm1\}$ then
$T^\sigma_s$ is disjoint from $T^\sigma_1$. For every
Gaussian-Kronecker automorphism $T:\xbm\to\xbm$ the iterations
$T^n$, $T^m$  are disjoint for any two distinct natural numbers $n, m$.


If $s$ is irrational then there exists a Gaussian-Kronecker flow
$\cT^\sigma$ such that $T^\sigma_s$ and $T^\sigma_1$ have a
non-trivial common factor.
\end{tw}

An importance of  Kronecker measures in ergodic theory follows from the following remarkable result of Foia\c{s} and Stratila
\cite{Fo-St} (see also \cite{Co-Fo-Si}, and  remarks on that
result in \cite{Le-Pa} and~\cite{Pa}):
\begin{equation}\label{fs1}
\begin{array}{l}
\text{If $(S_t)_{t\in \R}$ is an ergodic flow of a standard probability}\\
\text{Borel space $\ycn$, $f\in L^2\ycn$ is real and the spectral measure}\\
\text{$\sigma_f$ of $f$ is the symmetrization of a Kronecker measure,}\\
\text{then the (stationary) process $(f\circ S_t)_{t\in\R}$ is
Gaussian.}
\end{array}
\end{equation}
In \cite{Le-Pa}, any measure $\sigma$ satisfying the
assertion~(\ref{fs1}) of Foia\c{s}-Stratila theorem is called an
FS measure. Each Kronecker measure is a Dirichlet
measure\footnote{A probability Borel measure $\sigma$ on $\R$ is
Dirichlet, if (\ref{kro1}) is satisfied for $f=1$. From the
dynamical point of view, Dirichlet measures correspond to
rigidity: a flow $\cT$ is {\em rigid} if $T^{t_n}\to Id$ for some
$t_n\to\infty$.} \cite{Li-Po}, but as shown in \cite{Le-Pa}, there
are FS measures which are not Dirichlet measures (see Figure~\ref{fig1}).
\begin{figure}[ht]
\centering
\includegraphics[height=80pt]{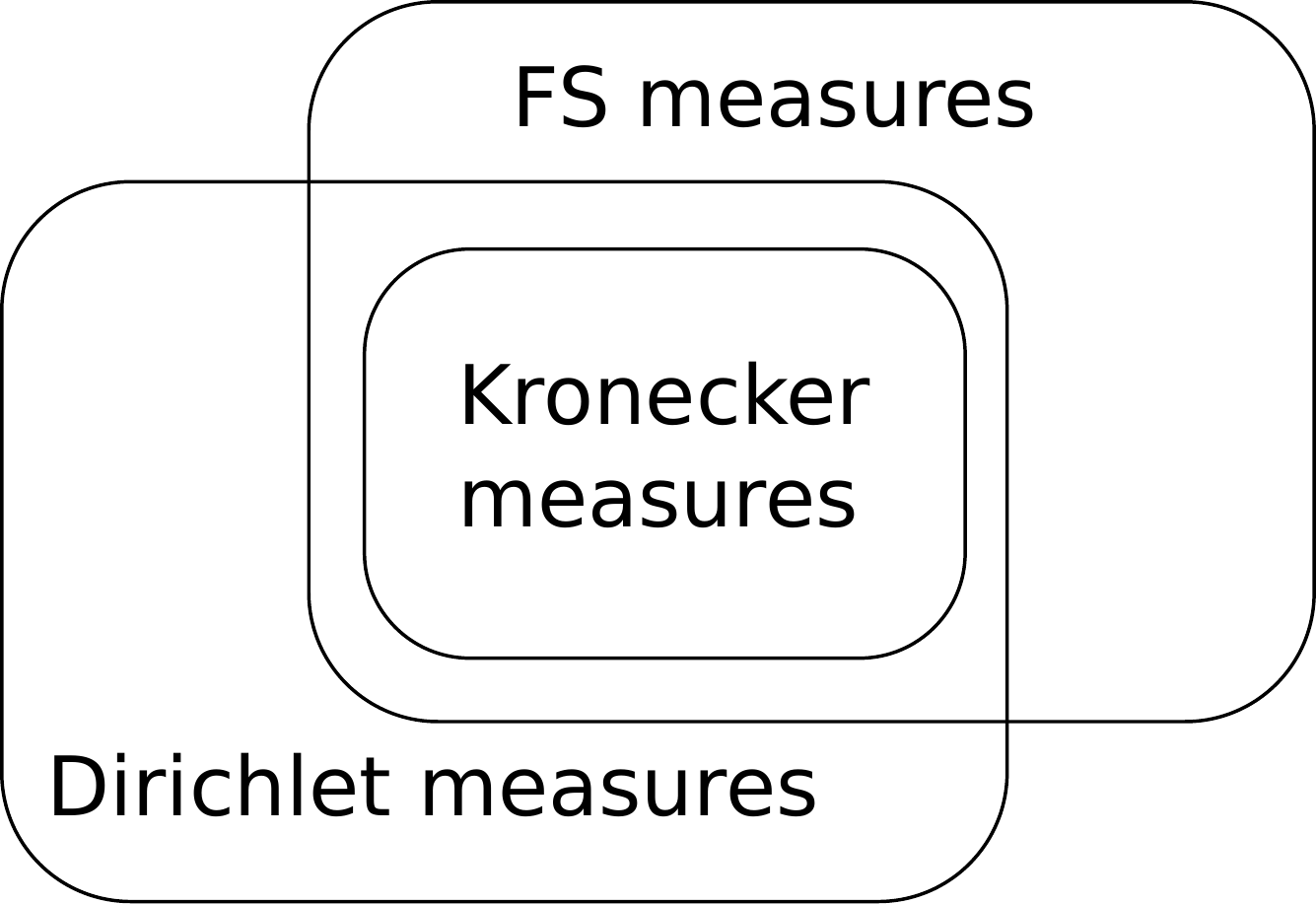}
\caption{Different classes of measures}
\label{fig1}
\end{figure}
Moreover, in
\cite{Pa}, it is announced that each continuous measure
concentrated on an independent Helson\footnote{$A\subset\R$ is
called Helson if for some $\delta>0$ and each complex Borel
measure $\kappa$ concentrated on $A$ the
$\sup_{t\in\R}\left|\int_{\R}e^{2\pi itx}\,d\kappa(x)\right|$ is
bounded away from the $\delta$-fraction of the total variation of
$\kappa$.} set is a Kronecker measure (for some examples in
\cite{Pa}, the resulting Gaussian flows have no non-trivial rigid
factors). We will strenghten Theorem~\ref{sw} to the following
result.

\begin{tw}\label{fs2}
Any symmetric countable group $G\subset \R^\ast$ can be  realized
as the group of self-similarities of a simple spectrum Gaussian
flow $\cT^\sigma$ with $\sigma$ being an FS measure.
\end{tw}

In particular, in connection with the forementioned question from
\cite{Le}, there is an FS measure for which the Gaussian flow has
simple spectrum but $\sigma$ is not concentrated on a
$\Q$-independent set. These are apparently the first examples of
FS measures which are not concentrated on $\Q$-independent Borel
sets but yield Gaussian flows with simple spectrum (cf.\
\cite{Le-Pa} and \cite{Pa}).

At the end of the note we  will discuss self-similarity properties
of Gaussian flows arising from a ``typical'' measure or from the
maximal spectral types of a ``typical'' flow (cf.\ the disjointness
results from \cite{Da-Ry}).

\begin{tw}\label{typical}
Assume that $0\leq a <b$. For a ``typical'' $\sigma\in\cP([a,b])$
the flow $\cT^\sigma$ is Gaussian-Kronecker such that for each
$|r|\neq |s|$ the flows $\cT^{\sigma_r}$ and $\cT^{\sigma_s}$ are
disjoint. In particular $I(\cT^\sigma)=\{\pm1\}$.

For a ``typical'' flow $\cT$ of a standard probability Borel space
$\xbm$, for its maximal spectral type $\sigma_{\cT}$ we have:
$\cT^{\sigma_{\cT}|_{\R_+}}$ has simple spectrum and for $|r|\neq |s|$ the
flows $\cT^{(\sigma_{\cT}|_{\R_+})_r}$ and $\cT^{(\sigma_{\cT}|_{\R_+})_s}$ are
disjoint. In particular $I(\cT^{\sigma_{\cT}|_{\R_+}})=\{\pm1\}$.
\end{tw}

\section{Notation and basic results}
Assume that $\cT=(T_t)_{t\in\R}$ is a
measurable\footnote{Measurability means that for each $f\in
L^2\xbm$ the map $t\mapsto f\circ T_t$ is continuous.}
measure-preserving flow acting on a standard probability Borel
space $\xbm$. It then induces a (continuous) one-parameter group
of unitary operators acting on $L^2\xbm$ by the formula
$T_tf=f\circ T_t$. By Bochner's theorem, the function $t\mapsto
\int_X T_tf\cdot\overline{f}\,d\mu$ determines the so called {\em
spectral measure} $\sigma_f$ of $f$ for which
$\widehat\sigma_f(t)= \int_X T_tf\cdot\overline{f}\,d\mu$,
$t\in\R$. Usually, one only considers spectral measures of $f\in
L^2_0\xbm$, that is, of elements with zero mean (the spectral
measure of the constant function $c$ is equal to
$|c|^2\delta_{0}$). Then $\sigma_f$  is a finite positive
Borel measure on $\R$. Among spectral measures there are maximal
ones with respect to the absolute continuity relation. Each such
maximal measure is called a {\em maximal spectral type measure}
and, by some abuse of notation, it will be denoted by
$\sigma_{\cT}$. We refer the reader to \cite{Ka-Th} and \cite{Le}
for some basics about spectral theory of unitary representations
of locally compact Abelian groups in the dynamical context.

Assume that $\cT$ is  ergodic and let $\cS=(S_t)_{t\in\R}$ be
another ergodic flow (acting on $\ycn$). Any probability measure
$\rho$ on $(X\times Y,\cB\otimes\cC)$ which is $(T_t\times
S_t)_{t\in\R}$-invariant and has marginals $\mu$ and $\nu$
respectively, is called a {\em joining} of $\cT$ and $\cS$. If,
additionally, the flow $((T_t\times S_t)_{t\in\R},\rho)$ is
ergodic then $\rho$ is called an {\em ergodic
joining}~\footnote{If $\cT=\cS$ then we speak about {\em
self-joinings}.}. The ergodic joinings are extremal points in the
simplex of all joinings. If the set of joinings is reduced to
contain only the product measure then one speaks about {\em
disjointness} of $\cT$ and $\cS$ \cite{Fu} and we will write
$\cT\perp\cS$. Similar notions appear when one considers
automorphisms. Note that whenever for some $t\neq0$, $T_t\perp
S_t$ then $\cT\perp \cS$. Note also that whenever
\begin{equation}\label{disj}
\text{$\cT$ is weakly mixing then $\cT\perp\cS$ if and only if $T_1\perp S_1$.}
\end{equation}
Indeed, if $T_1\not\perp S_1$ then there exists an ergodic
joining $\rho$ between them different than the product measure. Then,
$\rho\circ (T_r\times S_r)$ for $0\leq r<1$ has the same properties.
By disjointness of $\cT$ and $\cS$, $\int_0^1\rho\circ (T_r\times
S_r)\,dr=\mu\otimes\nu$. But $T_1$ is weakly mixing, so
$\mu\otimes\nu$ is an ergodic joining of $T_1$ and $S_1$, and
therefore $\rho\circ(T_r\times S_r)=\mu\otimes\nu$. We refer the
reader to \cite{Gl} for the theory of joinings in ergodic theory.

A flow $\cT$ is called  Gaussian if there is a $\cT$-invariant subspace
$\mathcal{ H}\subset L_0^2\xbm$ of the zero mean real-valued
functions such that all non-zero variables from $\mathcal{H}$ are
Gaussian and the smallest $\sigma$-algebra making all these
variables measurable equals $\cB$. A Gaussian flow is ergodic if
and only if the maximal spectral type on $\mathcal{H}$ is
continuous (and then $\cT$ is weakly mixing). Since Gaussian
variables are real, it is not hard to see that their spectral
measures are symmetric, that is, for $f\in\mathcal{ H}$,
$\sigma_f$ is invariant under the map $R_{-1}:x\mapsto -x$.

A standard way to obtain a (weakly mixing) Gaussian flow is to
start with a finite positive continuous Borel measure $\sigma$ on
$\R_+$. Consider the symmetrization
 $\widetilde{\sigma}=\sigma+(R_{-1})_\ast\sigma$~\footnote{In general, when $f$ is a
measurable map from $(X,\cB)$ to $(Y,\cC)$ and $\kappa$ is a
probability measure on $X$ then $f_\ast(\kappa)$ is the measure on
$Y$ defined by $f_\ast(\kappa)(C)=\kappa(f^{-1}(C))$.}. We let
${\mathcal V}=(V_t)_{t\in\R}$ denote the one-parameter group of
unitary operators on $L^2(\R,\widetilde{\sigma})$ defined by
$V_t(f)(x)=e^{2\pi i tx}f(x)$.  Then the correspondence
\begin{equation}\label{asd} f(x)\mapsto f(-x)\end{equation}
yields the unitary conjugation of $\mathcal{V}$ and its inverse.
Let $\xbm$ be a {\em Gaussian probability space}, that is, a
standard probability space together with an infinite dimensional,
closed, real and ${\mathcal B}$-generating subspace ${\mathcal
H}\subset L^2\xbm$ whose all non-zero variables are Gaussian. We
then consider ${\mathcal H}+i{\mathcal H}$, so called {\em complex
Gaussian space}, and define an isomorphic copy of ${\mathcal V}$
on it. It is then standard to show (see e.g.\ \cite{Le-Pa-Th},
Section~2) that $\mathcal{ V}$ has a unique extension to a
(measurable) flow ${\mathcal T}^{\sigma}=(T^{\sigma}_t)$ of $\xbm$
for which $U_{T^{\sigma}_t}|_{{\mathcal H}+i{\mathcal H}}=V_t$,
$t\in\R$. By the same token, the correspondence~(\ref{asd})
extends to an isomorphism of $\xbm$ which conjugates the Gaussian
flow and its inverse $(T^\sigma_{-t})_{t\in\R}$.

A Gaussian flow $\cT^\sigma$ is called Gaussian-Kronecker (FS resp.) if $\sigma$ is a continuous
Kronecker (FS resp.) measure. Following
\cite{Le-Pa-Th}, a Gaussian flow $\cT^\sigma$ (with the Gaussian
space $\mathcal{H}$) is called GAG if for each its ergodic
self-joining $\eta$ the space
$$
\{f(x)+g(y):f,g\in\mathcal{ H}\}$$ consists solely of Gaussian
variables  (the flow $(T^\sigma_t\times T^\sigma_t)_{t\in\R}$ is
then a Gaussian flow as well). We have \cite{Le-Pa-Th}

\begin{figure}[ht]
\centering
\includegraphics[height=100pt]{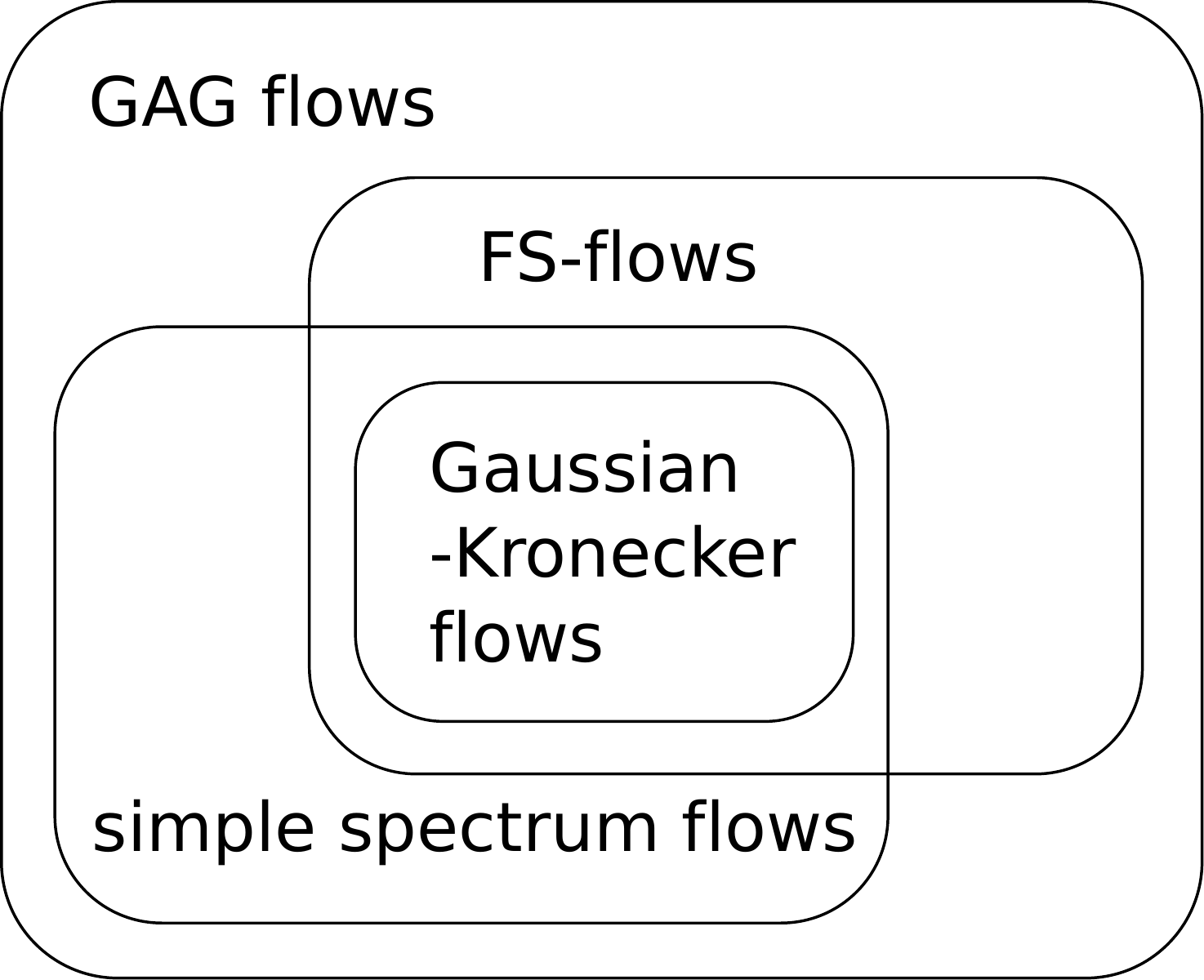}
\caption{Different subclasses of GAG flows}
\end{figure}

For all these classes of flows we have that  if $\cT^\sigma$ is in
the class, so is $\cT^{\sigma_s}$ for $s\neq 0$.

In general, Gaussian flows given by equivalent  measures are
isomorphic. It follows from \cite{Le-Pa-Th} that any isomorphism
between a GAG flow $\cT^\sigma$ and another Gaussian flow
$\cT^\nu$ is entirely determined by a unitary isomorphism of
restrictions of the unitary actions $(T^\sigma_t)_{t\in\R}$ and $(T^\nu_t)_{t\in\R}$
to their Gaussian subspaces. That is, in the GAG situation,
$\cT^\sigma$ are $\cT^\nu$ are isomorphic if and only if
$\sigma\equiv\nu$. If we apply that to $\sigma$ and $\sigma_s$ for
$s\in\R_+$ we will immediately get~(\ref{ni}) to hold (in the GAG
case).

We will now prove the following.

\begin{pr}\label{ml1}
Assume that $\cT^\sigma$ is GAG. Fix $s\neq0$. Then the sets of
self-joinings of $\cT^\sigma$ and of self-joinings of $T^\sigma_s$
are the same. (Hence ergodic self-joinings are also the same.) In
particular, the factors and the centralizer of the flow and of the
time $s$-automorphism are the same.
\end{pr}
\begin{proof}
This follows from the proof of Theorem~6.1 in \cite{Ju-Ru} which
asserts that such an equality of the sets of self-joinings takes place
whenever each ergodic self-joining of the flow is an ergodic
self-joining for the time-$s$ automorphism. In the GAG case, by
definition, such ergodic joinings  for the flow $\cT^\sigma$ are
Gaussian joinings, so they are automatically ergodic for the
$T^\sigma_s$ \cite{Le-Pa-Th}. \end{proof}

\begin{wn}\label{ml11}
Assume that $\cT^\sigma$ is GAG.  Then $T^\sigma_s$ is a GAG
automorphism for each $s\neq0$.
\end{wn}

\noindent We will also make use of the following  results.

\begin{tw}[\cite{Le-Pa-Th}]\label{lpt}
Assume that $\cT^\sigma$ is GAG and let $\cT^\eta$ be an arbitrary
Gaussian flow. Then $\cT^\sigma\perp\cT^\eta$ if and only if
$\widetilde{\sigma}\perp\widetilde{\eta}\ast\delta_r$ for each
$r\in\R$.
\end{tw}

\begin{pr}[\cite{Le-Pa}]\label{lp100}
If $\sigma_1$ and $\sigma_2$ are measures with the FS property and
$\cT^{\sigma_1}\perp\cT^{\sigma_2}$ then
$\sigma=\frac12(\sigma_1+\sigma_2)$ is also an FS measure.
\end{pr}

\section{Auxiliary lemmas}
Given a compact subset $X\subset\R$ denote by $\cP(X)$ the set of
all Borel probability measures concentrated on $X$ endowed with
the usual weak topology which is compact and metrizable: if
$\{f_n:\:n\geq1\}$ is a dense set in $C(X)$ then
\begin{equation}\label{mkj1}
d(\sigma,\eta)=\sum_{n=1}^\infty \frac1{2^n}\frac{\left|\int
f_n\,d\sigma-\int f_n\,d\eta\right|}{1+\left|\int
f_n\,d\sigma-\int f_n\,d\eta\right|}
\end{equation}
defines a metric compatible with the weak topology. Denote
$\cU(X)=\{f\in C(X):\: |f|=1\}$ which is a closed subset of $C(X)$
in the uniform topology, in particular $\cU(X)$ is a Polish space.

\begin{lm}\label{km100}
Assume that $X=[a,b]$. Let $\{h_0,h_1,\ldots,h_m\}\subset\R^\ast$
be a $\Q$-independent set. Then for each $f\in
\cU\left(\bigcup_{j=0}^mh_jX\right)$ and $\vep>0$
\begin{equation}\label{mkj2}
\begin{array}{c}
A_{f,\vep}(h_1,\ldots,h_m):=\\
\left\{\sigma\in \cP\left([a,b]\right):\:(\exists t\in\R)\;\left\|f-\xi_{t}\right\|_{L^2\left(\R,\frac1{m+1}\sum_{j=0}^m\sigma_{h_j}\right)}<\vep\right\}
\end{array}
\end{equation}
is open and dense in $\cP([a,b])$.
\end{lm}
\begin{proof}
The set $A_{f,\vep}(h_1,\ldots,h_m)$ is clearly open, so we need
to show its denseness in $\cP(X)$. Since discrete measures with a
finite number of atoms form a dense subset of $\cP(X)$ we take
$\nu=\sum_{s=1}^Na_s\delta_{y_s}$ with $y_s\in [a,b]$, $a_s>0$,
$s=1,\ldots,N$ and $\sum_{s=1}^N  a_s=1$ and fix $\delta>0$. All
we need to show is to find a subset
$\{x_1,\ldots,x_N\}\subset[a,b]$ such that $|x_s-y_s|<\delta$ for
$s=1,\ldots,N$ and such that the set
$$L:=\bigcup_{j=0}^m\left\{h_jx_1,\ldots,h_jx_N\right\}\quad \text{is $\Q$-independent}.$$
Indeed, in this case  by Kronecker's theorem, the set $L$ is a
finite Kronecker set, so the measure
$\frac1{m+1}\sum_{j=0}^m\left(\sum_{s=1}^Na_s\delta_{x_s}\right)_{h_j}$ is
Kronecker, whence belongs to $A_{f,\vep}(h_1,\ldots,h_m)$ and it
$\delta$-approximates $\nu$. To show that $x_1,\ldots,x_N$ can be
selected so that $L$ is $\Q$-independent, consider the algebraic
varieties of the form
$$
\left\{(z_1,\ldots,z_N)\in X^{\times N}:\: \sum_{j=0}^m\sum_{s=1}^Nk_{js}h_jz_s=0\right\}
$$
for some non-zero integer matrix $\left(k_{js}\right)$. Since
$$
\sum_{j=0}^m\sum_{s=1}^Nk_{js}h_jz_s=\sum_{s=1}^N\left(\sum_{j=0}^mk_{js}h_j\right)z_s$$
and $\sum_{j=0}^mk_{js}h_j\neq 0$ whenever
$(k_{0s},\ldots,k_{ms})\neq(0,\ldots,0)$ (and there are such
vectors since the matrix $(k_{js})$ is not zero), each such
variety has $N$-dimensional Lebesgue measure zero. Since there are
only countably many such varieties involved, we may discard the
union $S$ of them from $[a,b]^{\times N}$. Now, each choice of
$(x_1,\ldots,x_N)$ from
$(y_1-\delta,y_1+\delta)\times\ldots\times(y_N-\delta,y_N+\delta)\setminus
S$ satisfies our requirements.
\end{proof}

\begin{lm}\label{km101}
Given $H\subset \R^\ast_+$ a countable subset which is a
$\Q$-independent set, the set of continuous (Kronecker) measures
$\sigma\in\cP([a,b])$  for which the measure
\begin{equation}\label{km99}
\sum_{h\in H}a_h\sigma_h\quad \text{is a Kronecker measure (on $\R$)}
\end{equation}
for each choice of $a_h\geq0$, $\sum_{h\in H}a_h=1$, is a $G_\delta$ and
dense subset of $\cP([a,b])$.
\end{lm}
\begin{proof}
Denote by $\cP_c([a,b])$ the set of continuous measures which is a
$G_\delta$ and dense subset of $\cP([a,b])$.
Let $H=\{h_0,h_1,h_2,\ldots\}$.
For every $m\geq 0$ fix  a countable
dense family $\left\{g^{(m)}_i: i\geq 1\right\}\subset\cU\left(\bigcup_{i=0}^m h_i[a,b]\right)$.
Then, by Lemma~\ref{km100}, the set
\[
\cP_c([a,b])\cap\bigcap_{m=1}^\infty\bigcap_{i=1}^\infty\bigcap_{p=1}^\infty
A_{g^{(m)}_i,\frac1p}(h_1,\ldots,h_m)\]
is $G_\delta$ and dense in
$\cP([a,b])$ and it remains to show that this is precisely the set
of measures satisfying~(\ref{km99}). Indeed, given $m\geq1$, the
set
\[
\mathcal{K}_m(H):=\cP_c([a,b])\cap\bigcap_{i=1}^\infty\bigcap_{p=1}^\infty
A_{g^{(m)}_i,\frac1p}(h_1,\ldots,h_m)\]
is precisely the set of
continuous Kronecker measures $\sigma\in\cP([a,b])$ such that the
measure $\frac1{m+1}\sum_{i=0}^m\sigma_{h_i}$ is a Kronecker
measure (on the real line). Moreover, each measure absolutely
continuous with respect to a Kronecker measure is also a Kronecker
measure \cite{Le-Pa}. Therefore the set
$\mathcal{K}_m(H)$ is equal to the set of all
Kronecker measures $\sigma\in\cP([a,b])$ such that $\sum_{i=0}^m
b_i\sigma_{h_i}$ is Kronecker for arbitrary choice of $b_i\geq0$,
$\sum_{i=0}^mb_i=1$. Finally, for each $m\geq1$,
\[\frac1{\sum_{i=0}^ma_{h_i}}\sum_{i=0}^ma_{h_i}\sigma_{h_i}\ll \frac1{m+1}\sum_{i=0}^{m}\sigma_{h_i},\]
so if for each $m\geq1$ the measure
$\frac1{m+1}\sum_{i=0}^m\sigma_{h_i}$ is Kronecker, so is
$\sum_{h\in H}a_h\sigma_h$.
\end{proof}

\begin{uw}\label{TWKorner}
The idea of the above proofs is taken from a letter that has been
sent to us by T.W.\ K\"orner. In this letter, T.W.\ K\"orner shows
that given a transcendental number $h\in\R$,  for a ``typical''
(in the Hausdorff metric)  closed subset $K\subset[a,b]$ the set
$K\cup hK$ is Kronecker and uncountable. The proofs are the same since finite sets
are dense in the Hausdorff metric and if $h$ is transcendental
then given distinct $y_1,\ldots,y_N\in[a,b]$ and $\delta>0$ we can
find $q_i\in\Q$ so that for $x_i:=h^{2i}q_i$ we have
$|x_i-y_i|<\delta$ for $i=1,\ldots,N$ and clearly the set
$\{x_1,\ldots,x_N,hx_1,\ldots,hx_N\}$ is $\Q$-independent. It only
remains to notice that uncountable closed subsets are typical in
the Hausdorff metric.

Note also that using the proofs of Lemmas~\ref{km100}
and~\ref{km101}, given $H\subset\R^\ast_+$ a countable
multiplicative subgroup which is additively $\Q$-independent, we
obtain that a typical (with respect to the Hausdorff distance)
closed subset $K\subset[a,b]$ has the property that for each
finite subset $C\subset H$ the set $\bigcup_{h\in C} hK$ is
Kronecker, so the set $\bigcup_{h\in H}hK$ is a
$\Q$-independent $F_\sigma$-set.
\end{uw}

We  will also need the following ``compact $\Q$-independent set''
version of Lemma~\ref{km101}.

\begin{lm}\label{km102}
Assume that $K\subset\R$ is a compact uncountable set. Assume that
$H\subset\R^\ast_+$ is a countable set which is
additively $\Q$-independent. Assume moreover that the set
$\bigcup_{h\in H}hK$ is $\Q$-independent. Then the set of
continuous (Kronecker) measures $\sigma$ concentrated on $K$ for
which the measure
\begin{equation}\label{km99a}
\sum_{h\in H}a_h\sigma_h\quad \text{is a Kronecker measure}
\end{equation}
for each choice of $a_h\geq0$, $\sum_{h\in H}a_h=1$, is a $G_\delta$ and
dense subset of $\cP(K)$.\end{lm}
\begin{proof}
This follows from the proofs of Lemmas~\ref{km100}
and~\ref{km101}, where in addition the proof of Lemma~\ref{km100}
is simplified; indeed, for any choice of
$\{y_1,\ldots,y_N\}\subset K$ the set
$\bigcup_{j=0}^m\{h_jy_1,\ldots,h_jy_N\}$ is $\Q$-independent by
assumption (so we may take $x_i=y_i$ for $i=1,\ldots,N$).
\end{proof}

\begin{uw}\label{rem:restfun}
For any non-trivial compact interval $[a,b]\subset\R$ denote by
$\cP_c^{[a,b]}(\R)$ the set of measures $\nu\in\cP_c(\R)$ such
that $\nu([a,b])>0$. Since the map $\cP_c(\R)\ni\nu\mapsto
\nu([a,b])\in\R$ is continuous, the set $\cP_c^{[a,b]}(\R)$ is
open and dense in $\cP_c(\R)$. Let us consider the map
$\Delta=\Delta^{[a,b]}:\cP_c^{[a,b]}(\R)\to\cP_c([a,b])$ such that
$\Delta(\nu)$ is the conditional probability measure $\nu(\,\cdot\, |[a,b])$. This map is continuous and the preimage of
any dense subset of $\cP_c([a,b])$ is dense in
$\cP_c^{[a,b]}(\R)$. Indeed, let $A\subset\cP_c([a,b])$ be dense
and take any  $\nu \in \cP_c^{[a,b]}(\R)$. Then there exists a
sequence $(\widetilde{\nu}_n)_{n\leq 1}$ in $A$ such that
$\widetilde{\nu}_n\to \Delta(\nu)$ weakly. For every $n\geq 1$
define $\nu_n\in \cP_c^{[a,b]}(\R)$ so that the restriction of
$\nu_n$ to $[a,b]$ is $\nu([a,b])\widetilde{\nu}_n$ and the
measures $\nu_n$ and $\nu$ coincide on
$\R\setminus [a,b]$. Then $\Delta(\nu_n)=\widetilde{\nu}_n\in A$
and $\nu_n\to\nu$ weakly. Consequently, the preimage
$\Delta^{-1}A$ of any $G_\delta$ dense subset
$A\subset\cP_c([a,b])$ is $G_\delta$ dense in $\cP_c^{[a,b]}(\R)$.
\end{uw}

Before we  prove a certain disjointness property of Kronecker
measures, we will need the following general observation.

\begin{lm}\label{lem:malawieza}
Let $(X,\mathcal{B})$ be a standard Borel space and let
$\varphi:X\to X$ be a measurable map. Let $\sigma$ be a finite positive
continuous Borel measure on $X$ such that the map $\varphi:(X,\sigma)\to
(X,\varphi_*\sigma)$ is almost everywhere invertible. Assume that
$\sigma(\{x\in X:\:\varphi(x)=x\})=0$ and that the measures $\sigma$ and
$\varphi_*\sigma$ are not mutually singular. Then there exists a
measurable set $A\in\mathcal{B}$ such that $\sigma(A)>0$,
$\sigma(A\cap\varphi^{-1}A)=0$ and the measures $\sigma$ and
$\varphi_*\sigma$ restricted to
$A$ are equivalent.
\end{lm}

\begin{proof}
By assumption, there exists $Y\in\mathcal{B}$ such that
$\sigma(Y)>0$ and the measures $\sigma$ and $\varphi_*\sigma$
restricted to $Y$ are equivalent. It follows that  if $A\in\cB$,
$A\subset Y$, $\sigma(A)>0$, then the measures $\sigma$
and $\varphi_\ast\sigma$ restricted to $A$ are also equivalent.

\emph{Case 1.} Suppose that there exists $B\in\mathcal{B}$ such
that $B\subset Y$ and $\sigma(B\setminus \varphi(B))>0$. Set
$A:=B\setminus \varphi(B)$. Then $\sigma(A)>0$ and
$A\cap\varphi^{-1}A=(B\setminus \varphi(B))\cap
(\varphi^{-1}B\setminus B)=\emptyset$. Since $A\subset B\subset
Y$, our claim follows.

\emph{Case 2.} Suppose that for every $B\in\mathcal{B}$ with
$B\subset Y$ we have $\sigma(B\setminus \varphi(B))=0$. As
$\sigma$ and $\varphi_*\sigma$ restricted to $Y$ are equivalent,
it follows that
\begin{equation}\label{eq:zaw1}
0=\varphi_*\sigma(B\setminus
\varphi(B))=\sigma(\varphi^{-1}B\setminus B)\quad \text{ for every
} \quad  B\subset Y.
\end{equation}
We now show that there exists
$A\in\mathcal{B}$ such that $A\subset Y$, $\sigma(A)>0$ and
$\sigma(A\cap\varphi^{-1}A)=0$, which gives our assertion. Suppose
that, contrary to our claim, for every $A\in\mathcal{B}$ with
$A\subset Y$ the condition $\sigma(A)>0$ implies
$\sigma(A\cap\varphi^{-1}A)>0$. It follows that
\begin{equation}\label{eq:zaw2}
\sigma(B\setminus\varphi^{-1}B)=0\quad \text{ for every measurable
} \quad  B\in\mathcal{B}\text{ with }B\subset Y.
\end{equation}
Indeed,  otherwise for some $B$ as above,
$A:=B\setminus\varphi^{-1}(B)\subset Y$ would be of positive
$\sigma$-measure and since
\[(B\setminus\varphi^{-1}B)\cap\varphi^{-1}(B\setminus\varphi^{-1}B)=
(B\setminus\varphi^{-1}B)\cap(\varphi^{-1}B\setminus\varphi^{-2}B)=\emptyset,\]
and we would get $\sigma(A\cap
\varphi^{-1}A)=0$, a contradiction.

Now,  \eqref{eq:zaw1} combined with
\eqref{eq:zaw2} gives $\sigma(B\triangle\varphi^{-1}B)=0$ for
every $B\in\mathcal{B}$ with $B\subset Y$. It follows that
$\varphi(x)=x$ for $\sigma$-a.e. $x\in Y$, contrary to assumption.
\end{proof}

For any real  $s$ let $\theta_s:\R\to\R$, $\theta_s(t)=t+s$.
Recall that for every $n\in\Z$ and $z_1,z_2\in {\mathbf S^1}$ we
have
\begin{equation}\label{eq:expy1}
|z_1^n-z_2^n|\leq |n||z_1-z_2|.
\end{equation}

\begin{lm}\label{krzys}
Let $\sigma$ be a continuous Kronecker measure on $\R$. Then for every
$s\in\Q^*\setminus\{1\}$ and $r\in\R$ we have
$\sigma\perp\sigma_s *\delta_r$.
\end{lm}

\begin{proof}
Suppose that, contrary to our claim, there exists
$s\in\Q^*\setminus\{1\}$ and $r\in\R$ such that
$\sigma\not\perp\sigma_s *\delta_r$. Let $\varphi:=\theta_r\circ
R_s$. Then $\varphi:\R\to\R$ is an invertible map  with one fixed
point and $\sigma_s *\delta_r=\varphi_*\sigma$. By
Lemma~\ref{lem:malawieza}, there exists a Borel set $A_0\subset\R$
such that $\sigma(A_0)>0$, $\sigma(A_0\cap\varphi^{-1}A_0)=0$ and
the measures $\sigma$, $\varphi_*\sigma$ restricted to $A_0$ are
equivalent. Thus $\sigma(\varphi^{-1}A_0)>0$. Let $A_1, A_2\subset
A_0$ be disjoint Borel subsets such that
$\sigma(\varphi^{-1}A_1)>0$ and $\sigma(\varphi^{-1}A_2)>0$.

Let $s=q/p$ with $p$ and $q$ relatively prime integer numbers.
Choose  $z_0\in {\mathbf S^1}$ such that $z_0^q\neq 1$. Let us
consider the measurable map $f:\R\to {\mathbf S^1}$ such that
$f(x)=z_0$ if $x\in \varphi^{-1}A_2$ and $f(x)=1$ otherwise. Since
$\sigma$ is a Kronecker measure, there exists a sequence
$(t_n)_{n\in\N}$ of real numbers such that $\xi_{t_n}\to f$ in
$L^2(\R,\sigma)$. Thus $\xi_{t_n}\circ\varphi^{-1}\to
f\circ\varphi^{-1}$ in $L^2(\R,\varphi_*\sigma)$. Since
\begin{align*}
    g^0_n(x)&:=\chi_{A_0}(x)\left|\exp(2\pi i t_nx)-1\right|\leq \left|\xi_{t_n}(x)-f(x)\right|\\
    g^1_n(x)&:=\chi_{A_1}(x)\left|\exp(2\pi i t_ns^{-1}(x-r))-1\right|\leq \left|\xi_{t_n}(\varphi^{-1}x)-f(\varphi^{-1} x)\right|\\
    g^2_n(x)&:=\chi_{A_2}(x)\left|\exp(2\pi i t_ns^{-1}(x-r))-z_0\right|\leq \left|\xi_{t_n}(\varphi^{-1} x)-f(\varphi^{-1} x)\right|,
\end{align*}
it follows  that $(g^0_n)$ tends to zero in measure $\sigma$ and
the sequences $(g_n^1)$, $(g_n^2)$ tend to zero in measure
$\varphi_*\sigma$. As $\sigma\equiv\varphi_*\sigma$ on $A_0$ and
$A_1,A_2\subset A_0$, the sequences $(g_n^1)$, $(g_n^2)$ tend to
zero in measure $\sigma$, as well. Fix
\begin{equation}\label{def:eps}
    0<\vep<\frac{|z^{q}_0-1|}{2(|p|+|q|)}.
\end{equation}
Then there exist measurable sets $A_k'\subset A_k$, $k=0,1,2$  and
$n\in \N$ such that for $k=0,1,2$
\[\sigma(A_k\setminus
A_k')<\frac{1}{4}\min(\sigma(A_1),\sigma(A_2))\text{ and
}g_n^k(x)<\vep\text{ for all }x\in A_k'.
\]
Therefore for $k=1,2$ we have
\[\sigma(A_k\setminus A_0')\leq\sigma(A_0\setminus
A_0')<\frac{1}{4}\sigma(A_k)\text{ and }\sigma(A_k\setminus
A_k')<\frac{1}{4}\sigma(A_k),
\]
so $\sigma(A'_0\cap A'_k)>\sigma(A_k)/2>0$. Choose two real
numbers $x_1\in A'_0\cap A'_1$ and $x_2\in A'_0\cap A'_2$. Then
\begin{align*}
    & \left|\exp(2\pi i t_nx_1)-1\right|=g_n^0(x_1)<\vep, \quad
    \left|\exp(2\pi i t_n\frac{p}{q}(x_1-r))-1\right|=g_n^1(x_1)<\vep,\\
&   \left|\exp(2\pi i t_nx_2)-1\right|=g_n^0(x_2)<\vep,\quad
    \left|\exp(2\pi i t_n\frac{p}{q}(x_2-r))-z_0\right|=g_n^2(x_2)<\vep.
\end{align*}
In view of \eqref{eq:expy1},
\begin{align*}
&   \left|\exp(2\pi i t_np\,x_1)-1\right|<|p|\vep, \quad
    \left|\exp(2\pi i t_np(x_1-r))-1\right|<|q|\vep,\\
&   \left|\exp(2\pi i t_np\,x_2)-1\right|<|p|\vep,\quad
    \left|\exp(2\pi i t_np(x_2-r))-z^q_0\right|<|q|\vep.
\end{align*}
Hence
\[
    \left|\exp(2\pi i t_n p r)-1\right|<(|p|+|q|)\vep,\quad
    \left|\exp(2\pi i t_n p r)z_0^{q}-1\right|<(|p|+|q|)\vep,
\]
so
\[
    |1-z_0^{q}|<2(|p|+|q|)\vep,
\]
contrary to \eqref{def:eps}.
\end{proof}

Let us now consider the space $\cP(\R)$ of all Borel probability measures on $\R$ endowed with the weak topology.

By supp$(\sigma)$ we always mean the topological support of the
measure $\sigma$. Let us recall that
\begin{gather}\label{recall}
\begin{aligned}
&\text{if $\sigma\in\cP(\R)$ has ${\rm supp}(\sigma)=\R$}\\
&\text{then the set $\{\nu\in\cP(\R)\colon \nu\ll \sigma\}$ is dense in $\cP(\R)$.}
\end{aligned}
\end{gather}
Denote  by $\cP_c(\R)$ the set of all continuous members of
$\cP(\R)$ (this is a $G_\delta$ and dense subset of $\cP(\R)$).

The proof  of the lemma below is a slight modification of the
proof of Lemma 3.1 from \cite{Da-Ry}.
\begin{lm}\label{sw11} The set
\[
\mathcal{S}=\left\{\sigma\in \cP_c(\R)\colon\:
\sigma_s\perp\sigma\ast\delta_t\quad \text{for each} \;1\neq
s\in\R^\ast,\;t\in\R\right\}
\]
is $G_\delta$ and dense in $\cP(\R)$.
\end{lm}
\begin{proof}Denote by $\mathcal{I}$ the family of open subset of
$\R$ which are finite unions of open intervals.
Recall that for two measures $\sigma,\nu\in\cP(\R)$
\begin{equation}\label{cond:ort}
\sigma\perp \nu \Longleftrightarrow \forall_{n\in\N}\;\exists_{\mathcal{O}\in\mathcal{I}}\;\sigma(\mathcal{O})<1/n\quad\text{ and }\quad\nu(\mathcal{O})>1-1/n.
\end{equation}
For any compact rectangle  $I\times J\subset
(\R^\ast\setminus\{1\})\times\R$ denote by $\mathcal{V}(I\times
J)$ the set of all finite covers of $I\times J$ by compact
rectangles contained in $(\R^\ast\setminus\{1\})\times \R$. Notice
that for each open subset $\mathcal{O}\in\mathcal{I}$ the map
\begin{equation}\label{eq:map}
\cP_c(\R)\times \R^\ast\times\R\ni  (\sigma,s,r)\mapsto
\sigma_s\ast\delta_r(\mathcal{O})\in \R
\end{equation}
is  continuous. Therefore, given a compact rectangle $F\subset (\R^\ast\setminus\{1\})\times \R$ and an open subset $\mathcal{O}\in\mathcal{I}$ the map
\[
f_{F,\mathcal{O}}\colon \cP_c(\R)\ni \sigma\mapsto
\left(\sigma(\mathcal{O}), \max_{(s,r)\in F}\sigma_s\ast
\delta_r(\mathcal{O})\right) \in\R^2
\]
is continuous.
Let
\[
\widetilde{\cS}=\bigcap_{I\not\ni 1}\bigcap_{J}\bigcap_{n\in\N}
\bigcup_{\kappa\in\mathcal{V}(I\times J)}\bigcap_{F\in\kappa}
\bigcup_{\mathcal{O}\in\mathcal{I}}
f^{-1}_{F,\mathcal{O}}\left(\left(1-1/n, \infty\right)\times \left(-\infty,1/n\right) \right),
\]
where $I$ and $J$ run over closed intervals with rational endpoints. Then $\widetilde{\cS}$ is a $G_\delta$ set.

We claim that $\widetilde{\cS}=\cS$. Indeed,  let $\sigma\in \cS$.
Let $I\not\ni 1$ and $J\subset \R$ be compact intervals and  $n\in
\N$. By assumption and~\eqref{cond:ort}, for every $(s_0,r_0)\in I\times J$ there exists an open set
$\mathcal{O}_{s_0,r_0}\in\mathcal{I}$ such that
\[\sigma(\mathcal{O}_{s_0,r_0})>1-1/n\quad\text{ and }\quad
\sigma_{s_0}\ast \delta_{r_0}(\mathcal{O}_{s_0,r_0})<1/n.\]
Since the map \eqref{eq:map} is continuous, there exist open rectangles $U'_{s_0,r_0}\subset U_{s_0,r_0}\subset \R^2$ such that $(s_0,r_0)\in U'_{s_0,r_0}$ and a compact rectangle $F_{s_0,r_0}\subset(\R^\ast\setminus\{1\})\times \R$ satisfying $U'_{s_0,r_0}\subset F_{s_0,r_0}\subset U_{s_0,r_0}$ such that
\[
\sigma_{s}\ast \delta_{r}(\mathcal{O}_{s_0,r_0})<1/n\quad\text{ for all }(s,r)\in U_{s_0,r_0}.
\]
Since $I\times J$ is compact and $\{U'_{s,r}:(s,r)\in I\times J\}$
is its open cover, there exists a finite cover
$\kappa:=\left\{F_{s_1,r_1},\ldots,F_{s_k,r_k}\right\}$ of
$I\times J$. It follows that
\[
f_{F_{s_j,r_j},\mathcal{O}_{s_j,r_j}}(\sigma)\in (1-1/n, \infty)\times (-\infty,1/n)\quad\text{ for all }
\quad j=1,\ldots,k,
\]
thus $\sigma\in \widetilde{\cS}$.

Suppose that $\sigma\in \widetilde{\cS}$ and fix $s_0\in\R^*\setminus\{1\}$, $r_0\in\R$ and $n\in\N$.
Next choose  $I\not\ni 1$ and $J\subset \R$ compact intervals such that $(s_0,r_0)\in I\times J$.
By assumption, there exists a finite cover $\kappa\in \mathcal{V}(I\times J)$ such that for every $F\in\kappa$
there exists $\mathcal{O}_F\in \mathcal{I}$ with
\[
\sigma(\mathcal{O}_{F})>1-1/n\quad\text{ and }\quad
\sigma_{s}\ast \delta_{r}(\mathcal{O}_{F})<1/n\quad\text{ for all }(s,r)\in F.
\]
Choosing $F\in\kappa$ for which $(s_0,r_0)\in F$ and  applying \eqref{cond:ort} we have that
$\sigma$ and $\sigma_{s_0}*\delta_{r_0}$ are orthogonal, so $\sigma\in\cS$.

It remains to show that $\cS$ is dense.  To this end we use the
proof of Proposition 3.4 in \cite{Da-Ry}. Namely, in this
proposition there is a construction of a weakly mixing flow $\cT$
such that for a certain sequence of real numbers $u_k\to\infty$ we
have: for each $l\in\N$
\begin{equation}\label{dr1}
T_{-du_k}\to 10^{-l}\quad \text{for $d=1-10^{-l}$ and}
\end{equation}
\begin{equation}\label{dr2}
T_{-cu_k}\to 0\quad \text{uniformly in $c\in
[1,10^l]$}
\end{equation}
(the convergence takes place in the weak operator topology). It
follows that
\begin{equation}\label{dr0}
\sigma_{\cT_d}\perp \sigma_{\cT_c}\ast\delta_t
\end{equation}
for all $t\in\R$; indeed, (\ref{dr1}) and (\ref{dr2}) mean
respectively
\[\xi_{u_k}\to 10^{-l}\quad \text{weakly in $L^2(\R,\sigma_{\cT_d})$,}
\]
and
\[\xi_{u_k}\to 0\quad \text{weakly in $L^2(\R,\sigma_{\cT_c})$.}\]
It is easy to see that the latter condition implies
\[\xi_{u_k}\to 0\quad \text{weakly in $L^2(\R,\sigma_{\cT_c}\ast\delta_t)$}\]
for each $t\in\R$, and the mutual singularity~(\ref{dr0}) follows.

Now, in view of~(\ref{dr0}),
$\sigma_{\cT}\perp\sigma_{\cT_{c/d}}\ast\delta_{t/d}$, and since
in~(\ref{dr2}) $c$ can be replaced by $-c$, it follows that
$\sigma_{\cT}\in{\cS}$. It is also clear that $\cS$ is closed
under taking absolutely continuous measures. Since ${\rm
supp}\,\sigma_{\cT}=\R$~\footnote{This fact is well known for
$\Z$-actions, e.g.\ \cite{Na}, Chapter 3, and can be easily
rewritten using special representation of flows. See also the
proof of Theorem~A in \cite{Ri}.}, the result follows
from~(\ref{recall}).
\end{proof}

Recall also the following basic observation.
\begin{lm}\label{og}
Let $\overline{s}=(s_j)_{j\geq 1}$ be a sequence of positive numbers and let
$\overline{g}=(g_j)_{j\geq 1}$ be
a sequence of uniformly bounded continuous functions.
Then the set
\[
\mathcal{W}_{\overline{s},\overline{g}}=\left\{ \nu\in\mathcal{P}(\R)\colon \left(\exists\
{t_n\to \infty}\right)\left( \forall \  j\geq 1\right)\ \ \xi_{
s_j t_n } \to g_j \text{ weakly in } L^2(\R,\nu)\right\}
\]
is $G_\delta$ in $\mathcal{P}(\R)$.
\end{lm}

\begin{proof}
Let $(f_m)_{m\geq 1}$ be a sequence of continuous functions on
$\R$ uniformly bounded by $1$, which is linearly dense in
$L^2(\R,\nu)$ for every $\nu\in \cP(\R)$. Set
\[
\mathcal{R}(n,\vep)=
\left\{\mu\in \cP(\R):\ \sum_{m,j\geq 1}\frac{1}{2^{m+j}}
\left|\int_{\R}\left(e^{2\pi i s_j n x}- g_j(x)\right) f_m(x) \ d\mu(x)\right|<\vep\right\}.
\]
The set $\mathcal{R}(n,\vep)$ is open.
To complete the proof it suffices to notice that
\[
\mathcal{W}_{\overline{s},\overline{g}}=\bigcap_{\Q\ni \vep>0}\bigcap_{m\geq 1}\bigcup_{n\geq m} \mathcal{R}(n,\vep).
\]
\end{proof}

\begin{lm}\label{joa}
Let $H\subset \R^\ast_+$  be a countable multiplicative subgroup.
Then for a typical $\nu \in \cP(\R)$ the measure
$\eta:=\sum_{h\in H}a_h\nu_h$ (with $a_h>0$ and $\sum_{h\in
H}a_h=1$) yields a Gaussian flow $\cT^{\eta|_{\R_+}}$ with simple spectrum.
\end{lm}
\begin{proof}
Set $G=-H\cup H$ and let  $H=\{s_i:\: i\geq 0\}$ ($s_0=1$).
In~\cite{Da-Ry}, Danilenko and Ryzhikov constructed a rank-$1$ flow
$\cT$ preserving a $\sigma$-finite measure $\mu$ (the flow acts on
$\xbm$) such that if $\sigma=\sigma_{\cT}$ denotes its maximal spectral type
on $L^2\xbm$ then the Gaussian flow
\begin{equation}\label{SCSfornu}
\cT^{\left(\sum_{i\geq
1}\frac1{2^i}\sigma_{s_i}\right)|_{\R_+}} \text{ has simple spectrum. }
\end{equation}
To prove this, they used the following properties of $\cT$:
\begin{enumerate}
\item[a)]\label{1-a}
$T_{\sqrt{2}s}  \in WCP(T_s)$~\footnote{An operator $Q$ belongs to
the weak closure of powers ${\rm WCP}(R)$ if for an increasing
sequence $(m_j)$ of integers, $R^{m_j}\to P$ in the weak operator
topology.} for each $s\in H$,
\item[b)]\label{1-b}
$\frac{1}{q}I+\frac{q-1}{q}T_s \in WCP(T_s)$ for each $s\in H$ and
$q\in\N$,
\item[c)]
for each finite  sequence $s_1<s_2<\dots<s_k$ of elements of $H$
and each $1\leq l_0 \leq k$ there exists $t_j \to \infty$ such
that
\begin{enumerate}
\item[(i)]\label{1-c-1}
$T_{t_js_j}\to \frac{1}{2k}I \text{ if }1\leq l\leq k, l\neq l_0,$
\item[(ii)]\label{1-c-2}
$T_{t_js_{l_0}}\to \frac{1}{2k}T_{s_{l_0}}$.
\end{enumerate}
\end{enumerate}
Notice that the  conditions a), b) and c) can be expressed as
follows in terms of weak convergence of continuous and bounded
functions in $L^2(\R,\sigma)$:
\begin{itemize}
\item[a')]
for each $s\in H$ there exists a sequence $n_k\to \infty$ such that
$$\xi_{ s n_k } \to \xi_{\sqrt{2} s },$$
\item[b')]
for each $s\in H$ and $q\in\N$ there exists a sequence $n_k\to\infty$ such that
$$\xi_{s n_k}\to \frac{1}{q}+\frac{q-1}{q}\xi_{s},$$
\item[c')]
for each finite  sequence $s_1<s_2<\dots<s_k$ of elements of $H$
and each $1\leq l_0 \leq k$ there exists $t_j \to \infty$ such
that
\begin{itemize}
\item[(i)]
$\xi_{t_j s_j}\to \frac{1}{2k}$ if $1\leq l\leq k, l\neq l_0$,
\item[(ii)]
$\xi_{t_js_{l_0}}\to \frac{1}{2k}\xi_{s_{l_0}}$.
\end{itemize}
\end{itemize}
The arguments used in the proof of Theorem~4.4 in \cite{Da-Ry}
show that for each continuous probability measure $\sigma$ on $\R$
conditions a'), b') and c') imply the simplicity of spectrum of
the flow  $\cT^{(\sum_{k\geq 1}\frac1{2^k}\sigma_{s_k})|_{\R_+}}$.
Moreover, by Lemma~\ref{og},  the set of measures $\nu\in \cP(\R)$
satisfying these conditions is $G_\delta$. We will show now
that it is also dense in $\cP(\R)$. Notice that conditions a'),
b') and c') hold also in $L^2(\R,\nu)$ for any $\nu\ll \sigma$.
Since $\sigma_{\cT}$ is the maximal spectral type of a rank-$1$
infinite measure-preserving flow $\cT$, the Gelfand spectrum of the
corresponding Koopman representation is equal to $\R$. It follows that
the topological support of $\sigma_{\cT}$ is full and therefore the
result follows from \eqref{recall}.
\end{proof}

\section{Proofs of theorems}
\begin{proof}[Proof of Theorem~\ref{km10}]
(based on Lemmas~\ref{km101} and~\ref{sw11}.) Using these two
lemmas,  for a ``typical'' (continuous, Kronecker) measure $\sigma
\in\cP([a,b])$ we have (with $a_h>0$, and $\sum_{h\in H}a_h=1$)
$$
-H\cup H\subset I(\cT^\eta),$$
where $\eta:=\sum_{h\in H}a_h\sigma_h$ is a Kronecker measure
and moreover
\begin{equation}\label{oo}
\sigma_s\perp \sigma\ast\delta_t
\end{equation}
for each non-zero real $s\neq1$ and arbitrary $t\in\R$. All we
need to show is that when $s\notin -H\cup H$ then
$\eta_s\not\equiv \eta$. However if $s\notin H$ then even more is
true: $\eta\perp\eta_s\ast\delta_t$ for arbitrary $t\in\R$ and
$s\notin\{0,1\}$. It follows that
$$
\widetilde{\eta}_s\perp\widetilde{\eta}\ast\delta_t$$ for each
$s\notin-H\cup H$ and $t\in\R$. In view of Theorem~\ref{lpt}, it
follows that  $\cT^\eta$ is disjoint from $\cT^{\eta_s}$
(isomorphic to $\cT^{\eta}_s$) for $s\notin -H\cup H$. In
particular, $-H\cup H= I(\cT^\eta)$ and the result follows.
\end{proof}

\begin{proof}[Proof of Theorem~\ref{km10}]
(based on Lemma~\ref{km102}.) Given $H\subset\R^\ast_+$ a
multiplicative subgroup which is an additively  $\Q$-independent
set, in \cite{Fr-Le}, there is a construction of a perfect compact
set $K$ such that $\widehat{K}:=\bigcup_{h\in H}hK$ is independent
and for $\widetilde{K}:=-\widehat{K}\cup\widehat{K}$ the following
holds: $(r\widetilde{K}+t)\cap \widetilde{K}$ is countable
whenever $|r|\notin H$ and $t\in\R$ is arbitrary. Using
Lemma~\ref{km102} find a (continuous, Kronecker) measure
$\sigma\in\cP(K)$ such that $\eta:=\sum_{h\in H}a_h\sigma_h$ is a
Kronecker measure. Then $\eta$ is concentrated on $\widehat{K}$. All
we need to show is that if $|r|\notin H$, then the symmetrization
of $\eta_r$ is not equivalent to the symmetrization of $\eta$.
This is however clear, since the symmetrization of $\eta_r$ is a
continuous measure concentrated on $r\widetilde{K}$. As in the
previous proof we deduce  that for $s\notin -H\cup H$ we obtain
disjointness of the corresponding flows.
\end{proof}

\begin{proof}[Proof of Theorem~\ref{mkj}]
First notice that directly from Lemma~\ref{krzys}, it follows that whenever $\sigma$
is a Kronecker measure then for each $r_1,r_2\in\Q^\ast$, $r_1\neq
r_2$, we have
\begin{equation*}\label{ma20}
\sigma_{r_1}\perp\sigma_{r_2}\ast\delta_t\quad \text{for each
}\quad  t\in\R.
\end{equation*}
It follows
that
$\widetilde{\sigma}_{r_1}\perp\widetilde{\sigma}_{r_2}\ast\delta_t$
for all $t\in\R$, so by Theorem~\ref{lpt}, the Gaussian-Kronecker
flows $\cT^{\sigma_{r_1}}$ and $\cT^{\sigma_{r_2}}$ are disjoint.
In view of \eqref{disj}, it follows that $T^{\sigma_{r_1}}_1\perp
T^{\sigma_{r_2}}_1$, thus $T^{\sigma}_{r_1}\perp
T^{\sigma}_{r_2}$.

Now suppose that $T=T_\sigma:\xbm\to\xbm$ is a Gaussian-Kronecker automorphism,
i.e. $\sigma=\sigma_0+\overline{\sigma}_0$ for a continuous
Kronecker measure $\sigma_0\in\cP(\T)$.\footnote{$\T$ stands for $\{z\in\C\colon |z|=1\}$.} Denote by $\sigma'$ the
image of $\sigma_0$ via the map $\T\ni z\mapsto
\operatorname{Arg}(z)/2\pi\in[0,1)$. Then $\sigma'$ is a
continuous Kronecker measure on $\R$ such that
$(\xi_1)_*\widetilde{\sigma}'=\sigma$ and
$\widetilde{\sigma}'*\delta_m\perp \widetilde{\sigma}'$ for all
$m\in\N$. Denote by $\mathcal{H}$ the Gaussian space of the flow
$\cT^{\sigma'}$. Then the Koopman operator of $T^{\sigma'}_1$ has
simple spectrum on $\mathcal{H}$ and its spectral type is
$(\xi_1)_*\widetilde{\sigma}'=\sigma$, see Appendix in
\cite{Le-Pa:conv}. Since the spectral type of $\zeta_1$ (with
respect to $T^{\sigma'}_1$) is
$(\xi_1)_*\widetilde{\sigma}'=\sigma$, it follows that $\zeta_1\circ
(T^{\sigma'}_1)^n$, $n\in \Z$, span the space $\mathcal{H}$.
Thus $T^{\sigma'}_1$ is isomorphic to $T_{\sigma}$. By the first
assertion of the theorem, it follows that $T^n_{\sigma}$ is disjoint
from $T^m_{\sigma}$ for any pair of distinct natural numbers.

In order to prove the second part  of the theorem note that if $s$
is irrational then the set $\{1,s\}$ is $\Q$-independent, so by
Lemma~\ref{km101} we can find a (continuous, Kronecker) measure
$\sigma\in\cP([a,b])$ such that $\eta:=\frac12(\sigma+\sigma_s)$
is a Kronecker measure. Since $\sigma_s\ll\eta$ and
$\sigma_s\ll\eta_s$ the Gaussian-Kronecker flows $\cT^\eta$ and
$\cT^{\eta_s}$ have a common non-trivial (Gaussian) factor. Its
time one map is a common non-trivial factor of $T^\eta_1$ and
$T^{\eta_s}_1$ and it remains to notice that the Gaussian
automorphism $T^{\eta_s}_1$ is isomorphic to $T^\eta_s$.
\end{proof}

\begin{proof}[Proof of Theorem~\ref{fs2}]
Let $H=G\cap \R^*_+$ and let $(a_h)_{h\in H}$ be positive numbers
such that $\sum_{h\in H}a_h=1$. By Lemmas~\ref{sw11}, \ref{joa}
and Lemma~\ref{km101} (applied to $H=\{1\}$) combined with
Remark~\ref{rem:restfun}, there exists $\nu'\in\cP_c(\R)$ such that
\begin{itemize}
\item[(i)] $\nu'_s\perp \nu'\ast\delta_t$ for all $s\in \R^*\setminus\{1\}$ and $t\in
\R$;
\item[(ii)] the Gaussian flow $\cT^{(\sum_{h\in H}a_h \nu'_s)|_{\R_+}}$ has simple spectrum
\item[(iii)] $\nu:=\Delta(\nu')\in \cP_c([a,b])$ is a Kronecker measure
\end{itemize}
(in fact, for a ``typical'' $\nu'\in\cP_c(\R)$ the properties (i)-(iii) hold).
Since the conditions (i) and (ii) hold also for any measure
absolutely continuous with respect to $\nu$, the Kronecker measure
$\nu$ satisfies (i) and (ii) as well. Therefore, setting
$\sigma:=\sum_{h\in H}a_h \nu_s$, by (ii), the Gaussian flow
$\cT^{\sigma}$ has simple spectrum. The same argument as in the
proof of Theorem~\ref{km10} shows that (i) together with (ii)
imply $I(\cT^\sigma)=-H\cup H$ and $\cT^{\nu_s}\perp\cT^{\nu_r}$
whenever $|r|\neq|s|$. Each Kronecker measure $\nu_h$, $h\in H$ is
an FS measure so, by Proposition~\ref{lp100}, it follows that
$\sigma=\sum_{h\in H}a_h \nu_s$ is an FS measure~\footnote{We use
here the elementary fact that the $L^2$-limit of a sequence of
Gausian variables remains Gaussian.}, which completes the proof.
\end{proof}

\begin{proof}[Proof of Theorem~\ref{typical}]
The  first part follows  from Lemma~\ref{sw11} along the same
lines as the first proof of Theorem~\ref{km10} (for $H=\{1\}$).

In view of Corollary~2 in \cite{Le-Pa:conv}, a typical flow $\cT$
has the SC property,\footnote{The SC property means that if we set $\sigma=\sigma_{\cT}$ then for each $n\geq 2$ the conditional measures of the disintegration of $\sigma^{\otimes n}$ over $\sigma^{\ast n}$ via the map $\R^n\ni(x_1,\dots,x_n)\mapsto x_1+\dots+x_n\in\R$ are purely atomic with $n!$ atoms.} which is equivalent to the fact that $\cT^{\sigma_{\cT}}$
has simple spectrum. In particular, it implies that $\cT^{\sigma_{\cT}}$ is GAG.

In  order to prove that $\sigma_{\cT}\perp (\sigma_{\cT})_s\ast\delta_r$,
$s\in\R^*\setminus\{1\}$, $r\in\R$ for a typical flow $\cT$ we follow the
proof of Theorem 3.2 from \cite{Da-Ry} (using Lemma~\ref{sw11} and
the existence of a flow satisfying~(\ref{dr0})). Since $\cT^{\sigma_{\cT}}$ is GAG
for a typical flow $\cT$, by Proposition~\ref{lpt}, it follows that $\cT^{(\sigma_{\cT})_s}$
and $\cT^{(\sigma_{\cT})_r}$ are disjoint wherever $|r|\neq|s|$.
\end{proof}

\paragraph{\underline{Question}}
Is there a Kronecker measure $\sigma\in\cP(\R_+)$ such that $I(\cT^\sigma)$ is uncountable?
\\
This question is to be compared with Ryzhikov's question whether there is a weakly mixing, non-mixing flows with uncountable group of self-similarities, see~\cite{Da}, Problem (1).

\end{document}